\documentclass[10.5pt,oneside,english]{amsart}
  \usepackage{amsmath}
  \usepackage{amsfonts}
  \usepackage{amssymb}
  \usepackage[all]{xy}
  \usepackage{verbatim}  
  \usepackage{longtable}
  \usepackage{stmaryrd}   
  \usepackage{mathrsfs}   
  \usepackage{enumitem}
  \usepackage{graphicx}
  \usepackage{url}
  \usepackage{color}
  \usepackage{mathtools}
  \usepackage[center,loose,nooneline]{subfigure}

\newtheorem{theorem}{Theorem}
\newtheorem{lemma}[theorem]{Lemma}

 \newcommand{\R}{\mathbb{R}}

 \newcommand{\Hyp}{\mathbb{H}}

 \newcommand{\Hhh}{\mathscr{H}}

 \newcommand{\dist}{\mathrm{dist}}
 \newcommand{\atanh}{\mathrm{atanh}}

  \newcommand{\woz}{\backslash\{0\}}
  \newcommand{\wo}{\backslash}

 \newcommand{\lm}{\varnothing}

  \newcommand{\eucl}{\text{\scalebox{0.8}{\,$\mathbb{E}$}}}  
  \newcommand{\poin}{\text{\scalebox{0.8}{\,$\mathcal{P}$}}}  
   \newcommand{\klein}{\text{\scalebox{0.8}{\,$\mathcal{K}$}}}

\begin{document}

\title{Projection theorems in hyperbolic space}
\author{Zolt\'an M. Balogh and  Annina Iseli}

\address {Mathematisches Institut,
Universit\"at Bern,
Sidlerstrasse 5,
CH-3012 Bern,
Switzerland}

\email{zoltan.balogh@math.unibe.ch}
\email{annina.iseli@math.unibe.ch}


\thanks{ This research was supported by the Swiss National Science Foundation Grant Nr.\ 00020 165507.}

\begin{abstract}
We establish Marstrand-type projection theorems for orthogonal projections along geodesics onto $m$-dimensional subspaces of hyperbolic $n$-space by a geometric argument. Moreover, we obtain a Besicovitch-Federer type characterization of purely unrectifiable sets in terms of these hyperbolic orthogonal projections. \end{abstract}

\maketitle

\section{Introduction}

Marstrand's theorem~\cite{Marstrand1954} states that given a Borel set $A\subset \R^2$, for almost every line $L$ the orthogonal projection of $A$ onto $L$ is a set of Hausdorff dimension equal to the minimum of $1$ and the Hausdorff dimension of $A$. This result has marked the start of a large series of results in the same spirit. In particular, Marstrand's theorem has been sharpened and generalized to higher dimensions by Kaufman~\cite{Kaufman1968}, Falconer~\cite{Falconer1982}, and Mattila~\cite{Mattila1975}. Also, similar problems have been studied in various other settings such as Heisenberg groups \cite{BDCFMT2013} \cite{BFMT2012} \cite{Hovila2014} and normed spaces~\cite{BaloghIseli2018}~\cite{Iseli_Arx2018}, as well as for radial projections in \cite{Orponen_Arx2017}, different notions of measure and dimension \cite{FalcHow1996} \cite{FalcHow1997} \cite{FrO2017}, and restricted families of projections~\cite{FO2014}~\cite{Chen_Arx2017}~\cite{OrpVen_Arx2017} (and references therein).
In this paper, we prove Marstrand-type projection theorems as well as a Besicovitch-Federer-type projection theorem (i.e. a characterization of purely unrectifiable sets in terms of projections) for orthogonal projection along geodesics in hyperbolic $n$-space. In particular, we generalize previous results of the authors~\cite{BaloghIseli2016} to higher dimensions. An extended introduction to the topic is provided in this previous work. For a more exhaustive background on projections theorems in various setting we recommend the recent survey article~\cite{Mattila_Arx2017} and the references therein.\smallskip

By $\Hyp^n$ denote the hyperbolic $n$-space and by $d$ the hyperbolic metric on $\Hyp^n$. Fix a base point $p\in \Hyp^n$ and identify the tangent plane $T_p\Hyp^n$ with $\R^n$. Now, consider the exponential mapping $\exp_p:\R^n\to \Hyp^n$ at $p$. 
 Note that for every $m$-plane $V$ (i.e. $m$-dimensional linear subspace of $\R^n$) the image $\exp_p(V)$ is a 
 geodesically convex $m$-dimensional submanifold of $\Hyp^n$ that is isometric to~$\Hyp^m$.
Since $\Hyp^n$ is a simply connected Riemannian manifold of constant sectional curvature equal to $-1$, for all $x\in \Hyp^n$, there exists a unique point $q\in \exp_p(V)$ such that \[\dist(x,\exp_p(V))=d(x,q).\] Define the projection of $\Hyp^n$ onto the hyperbolic $m$-plane $\exp_p(V)$ by
\[P_V:\Hyp^n\to \exp_p(V), \ P_V(x)=q.\]
As standard arguments show (see Proposition 2.4 in \cite{BH1999}), for all $x\in \Hyp^n$ and all $m$-planes $V$ the geodesic segment $[x,P_V(x)]$ intersects $\exp_p(V)$ orthogonally in the point $P_V(x)$. Therefore, we will refer to the collection of mappings $P_V:\Hyp^n\to \exp_p(V)$, for $m$-planes $V$, as the family of orthogonal projections (along geodesics) onto $m$-planes in $\Hyp^n$. \smallskip

It is known that the projections $P_V: \Hyp^n\to \exp_p(V)$ are $1$-Lipschitz (i.e. distance non-increasing) with respect to the hyperbolic metric $d$, and hence {$\dim P_V(A)\leq \dim A$}, for all sets $A\subseteq \Hyp^n$ and all {$m$-planes}~$V$. Moreover, the facts that $P_VA\subset V$ and {$\dim V=m$} imply that $\dim P_VA\leq m$ for all {$m$-planes}~$V$. This yields the same upper bound $\dim P_V A\leq \min\{m, \dim A\}$ as in the Euclidean setting. It is therefore a natural question whether the generic lower bounds for $\dim P_V A$ is the same as well, i.e. whether Marstrand-type projection theorems generalize to the hyperbolic setting.\smallskip

We call the family of all $m$-planes $V$ in $\R^n$ the Grassmannian of $m$-planes (in $\R^n$) and denote it by $G(n,m)$. The Grassmannian $G(n,m)$ carries a natural measure $\sigma_{n,m}$ that is induced by the Haar measure on $O(n)$ via the group action of $O(n)$ on $G(n,m)$; see \cite{Mattila1995}, Chapter~3. 
Moreover, the Grassmannian can be smoothly parametrized by local charts in $\R^K$, where $K=(n-m)m$; see~\cite{AnninaPhD}, Section 2.3. This yields a notion of zero sets for the $s$-dimensional Hausdorff measure $\Hhh^s$, $s>0$, and of Hausdorff dimension $\dim$ of subset of $G(n,m)$.\smallskip

The following Marstrand-type theorem is a main result of this paper. It can be considered an analog of results in Euclidean space due to Marstrand~\cite{Marstrand1954}, Kaufman \cite{Kaufman1968}, Falconer~\cite{Falconer1982}, Mattila~\cite{Mattila1975} and Peres-Schlag~\cite{PS2000}.
\begin{theorem}\label{thm_Marstrand}
For the family of orthogonal projections $P_V:\Hyp^n\to \exp_p(V)$, $V\in G(n,m)$, onto $m$-planes in $\Hyp^n$
and for all Borel sets $A\subseteq \Hyp^n$, the following hold.
\begin{enumerate}[label={(\arabic*)}, topsep=3pt, itemsep=3pt, leftmargin=30pt]
\item If $\dim A \leq m$, then 
\begin{enumerate}[topsep=2pt, itemsep=2pt]
\item  $\dim (P_V A)= \dim A$ for $\sigma_{n,m}$-a.e.\ $V\in G(n,m) $,
\item For $0<\alpha\leq\dim A$,\\
$\dim(\{V\in G(n,m) :  \dim(P_V A)<\alpha\})\leq (n-m-1)m+\alpha.$
\end{enumerate}
\item If $\dim A > m$, then 
\begin{enumerate}[topsep=2pt, itemsep=2pt]
\item  $\Hhh^m (P_V A)>0$ for $\sigma_{n,m}$-a.e.\ $V\in G(n,m) $,
\item $\dim(\{V\in G(n,m) : \mathscr{H}^m(P_V A)=0\})\leq (n-m)m+m-\dim A$.
\end{enumerate}
\item If $\dim A>2m$, then
\begin{enumerate}[topsep=2pt, itemsep=2pt]
\item  $P_V A$ has non-empty interior in $V$ for $\sigma_{n,m}$-a.e.\ $V\in G(n,m)$,
\item  $\dim(\{V\in G(n,m):  (P_VA)^\circ \neq \lm \})\leq (n-m)m+2m -\dim A.$
\end{enumerate}
\end{enumerate} 
\end{theorem}

We will prove Theorem~\ref{thm_Marstrand} by a comparison argument. Namely, we will define a self-map of the unit ball that by conjugation transforms hyperbolic orthogonal projections (displayed in the Poincar\'e model) into Euclidean orthogonal projections; see Section~\ref{sec_proofs}. 
The same arguments will allow us to establish a Besicovitch-Federer-type \cite{Besic1939}\cite{Federer1947} characterization of purely $m$-unrectifiable subsets of $\Hyp^n$. Recall that a subset $A$ of a metric space $X$ is called $m$-rectifiable if there exist at most countably many Lipschitz mappings $f_i:\R^m\to X$ such that 
\[\Hhh^m\Big(A\, \wo \bigcup f_i(\R^m)\Big)=0.\]
On the other hand, a set $F\subseteq X$ is called purely $m$-unrectifiable, if $\Hhh^m(F\cap A)=0$ for every $m$-rectifiable set $A\subseteq X$.
\begin{theorem}\label{thm_BesFed}
A set $A\subseteq \Hyp^n$ with $\Hhh^m(A)<\infty$ is purely $m$-unrectifiable if and only if for $\sigma_{n,m}$-a.e.\ $V\in G(n,m) $, we have $\Hhh^m(P_V(A))=0.$
\end{theorem} 
The Euclidean version of this result is sometimes also referred to as the Besicovitch-Federer projection theorem; see Theorem 18.1 in~\cite{Mattila1995}.

\section{Proofs of Theorem \ref{thm_Marstrand} and Theorem \ref{thm_BesFed}}\label{sec_proofs}

First, we recall some preliminaries on hyperbolic geometry and fix the notation used in the sequel. For a more detailed account on hyperbolic geometry as it is used here, we recommend the textbooks \cite{BH1999} and \cite{BenedettiPetronio}.\smallskip

Consider the Poincar\'e model of hyperbolic $n$-space $\Hyp^n$, that is, the metric space $(D^n, d_\poin)$ where 
$D^n:=\{x\in \R^n:|x|<1\}$ is the open unit ball in $\R^n$ and
and the Poincar\'e metric $d_\poin$ is given by
\[d_\poin(x,y)=2\, \atanh\left(\frac{|x-y|}{(1-2\langle x,y\rangle+|x|^2+|y|^2)^\frac{1}{2}}\right).\]
for all $x,y\in D^n$.\smallskip

Let $\Gamma$ be a circle in $\R^n$ that intersect $\partial D^n$ orthogonally. Then $\Gamma\cap D^n$ is a hyperbolic geodesic in the Poincar\'e model $(D^n,d_\poin)$. The same holds for $L\cap D^n$ for $L\in G(n,1)$. Conversely, every geodesic of hyperbolic space displayed in the Poincar\'e model is distance minimizing with respect to $d_\poin$ and is either of the type $\Gamma\cap D^n$ or $L\cap D^n$. Moreover, the Poincar\'e model is known to be a conformal model of hyperbolic space, i.e., the angle in which two curves in hyperbolic $n$-space intersect equals the Euclidean angle in which their representatives in $(D^n,d_\poin)$ intersect. This makes the 
Poincar\'e model a natural choice for studying orthogonal projections of hyperbolic $n$-space.
\begin{figure}[h] 
\begin{center}
\def\svgwidth{200pt}
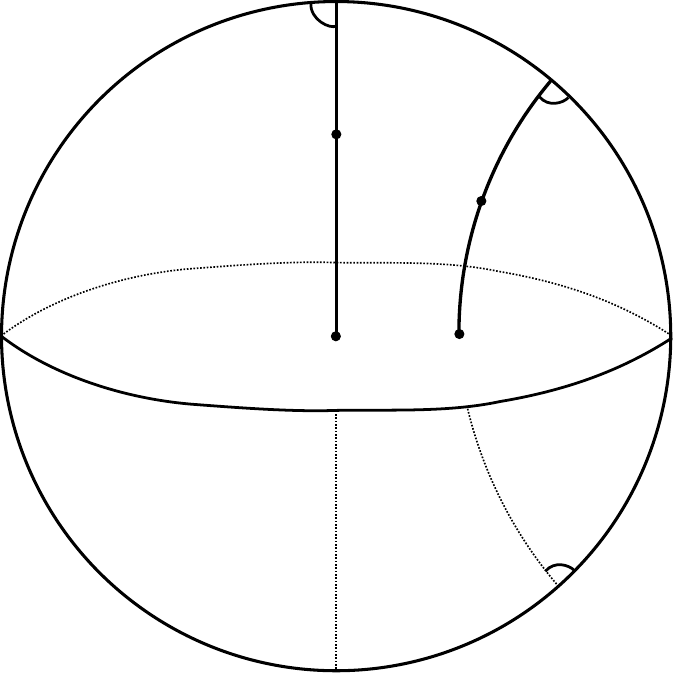
\end{center}
\caption{The projection $P_V^{\protect\poin} :D^3\to D^3\cap V$.} \label{fig_proj_hypn}
\end{figure}    

Choose $0$ to be the representative of the base point $p\in \Hyp^n$ in the model $(D^n, d_\poin)$. This choice is made without loss of generality since hyperbolic space is homogeneous with respect to its group of isometries. Then, for all $V\in G(n,m)$, the hyperbolic $m$-plane $\exp_p(V)$ corresponds to the $m$-dimensional disc $V\cap D^n$ in the model $(D^n,d_\poin)$. 
For each $V\in G(n,m)$, define \[P^\poin_V:D^n\to V\cap D^n\] to be the closest-point projection onto $V\cap D^n$ with respect to the metric $d_p$; see Figure~\ref{fig_proj_hypn}. 
Then, the family of hyperbolic orthogonal projections $P_V:\Hyp^n\to \exp_p(V)$, $V\in G(n,m)$, can be viewed as $P^\poin_V:D^n\to V\cap D^n$, $V\in G(n,m)$. Moreover, by conformality of the Poincar\'e model $(D^n,d_\poin)$, the family the projections $P^\poin_V:D^n\to V\cap D^n$ are orthogonal projections along geodesics in $(D^n,d_\poin)$.\smallskip

Now, consider the mapping $\Psi:D^n\to D^n$, defined by 
\[\Psi(x):=\frac{\tanh(\frac{1}{2}\atanh |x|)}{|x|}x,\]
for $x\in D^n\woz$, and $\Psi(0)=0$.
Notice that $\Psi$ is a bijection with inverse defined by
\[\Psi^{-1}(y)= \frac{\tanh(2\atanh |y|)}{|y|}y\] for $x\in D^n\woz$, and $\Psi^{-1}(0)=0$.
One can check that $\Psi$ maps every geodesic $\Gamma\cap D^n$ (where either $\Gamma\in G(n,1)$ or $\Gamma$ is a circle that intersects $\partial D^n$ orthogonally) to the Euclidean line segment that connects the endpoints $p_1,p_2$ of $\Gamma\cap D^n $; see Figure~\ref{fig_Psi_hyp}. 
 
\begin{figure}[h] 
\begin{center}
\def\svgwidth{200pt}
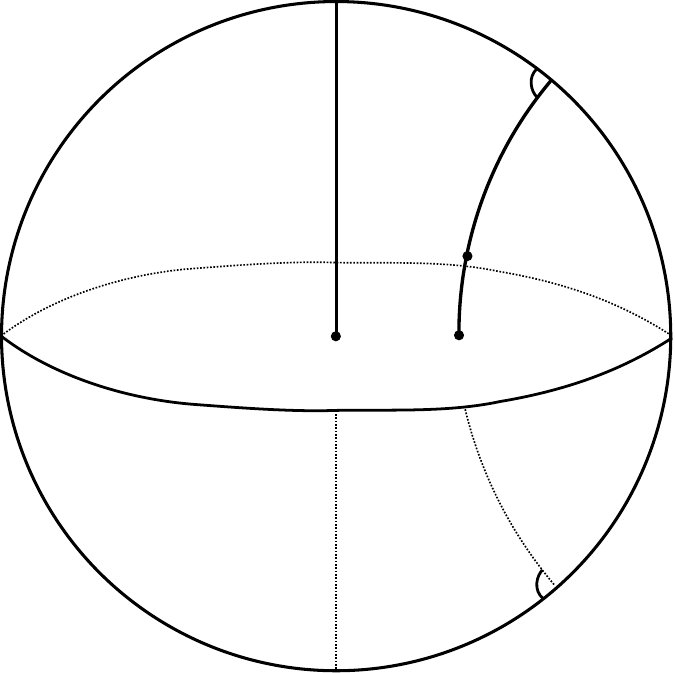
\end{center}
\caption{The mapping $\Psi: D^3\to D^3$ where $\Gamma$ is a geodesic in $(D^3,d_{\protect\poin})$ } \label{fig_Psi_hyp}
\end{figure}   

The metric space $(D^n,d_\klein)$ where $d_\klein(x,y):=d_\poin(\Psi^{-1}(x),\Psi^{-1}(y))$, for all {$x,y\in D^n$}, is often called the Klein model or the projective model of hyperbolic space; see~\cite{BenedettiPetronio} for details. Note that the Klein model is not a conformal model of hyperbolic space. However, if $\Gamma_1$ and $\Gamma_2$ are representatives of hyperbolic geodesics in $(D^n,d_\klein)$ and if $0\in \Gamma_1$, then the respective geodesics in hyperbolic space intersect orthogonally if and only if $\Gamma_1$ and $\Gamma_2$ intersect orthogonally in the Euclidean sense in $D^n$. \smallskip

The symmetry of $\Psi$ yields the following relation between the orthogonal projections in the Poincar\'e model and Euclidean orthogonal projections.

\begin{lemma}\label{lem_Psi_and_proj_1} For all $V\in G(n,m)$, the following holds: $P_V^\poin=\Psi^{-1}\circ P_V^\eucl\circ \Psi$.
\end{lemma}

\begin{proof}
Let $x\in D^n$ and $V\in G(n,m)$. By $\Gamma$ denote the circular arc in $D^n$ that is perpendicular to $V$ and $\partial D^n$ and contains $x$. Then, by definition, 
$P_V^\poin(x)$ is the unique intersection point of $V$ and $\Gamma$. Since $\Gamma$ intersects~$V$ orthogonally, the set $\Gamma\cap \partial D^n=\{p_1,p_2\}$ is symmetric under the reflection through~$V$.  Thus, the line segment $\Psi(\Gamma)$ connecting $p_1$ and $p_2$ intersects~$V$ orthogonally; see Figure~\ref{fig_Psi_hyp}.
By definition, $\Psi(x)$ is the unique intersection point of~$\Gamma$ with the ray that emerges from the origin and goes through $x$ within~$D^n$. Then, since $\Psi(x)\in \Psi(\Gamma)$, and $\Psi(\Gamma)$ intersects $V$ orthogonally, $P_V^\eucl(\Psi(x))$ is the point where $\Psi(\Gamma)$ intersects $V\cap D^n$. On the other hand, $\Psi(P^\poin_V(x))$ is the intersection point of~$\Psi(\Gamma)$ and the ray that emerges from the origin and passes through $P^\poin_V(x)$. However, this intersection point is exactly $P_V^\eucl(\Psi(x))$; see Figure~\ref{fig_Psi_hyp}.\end{proof}

\begin{proof}[Proof of Theorem~\ref{thm_Marstrand} and Theorem~\ref{thm_BesFed}]
Note that the restriction of the mapping $\Psi:D^n\to D^n$ to $D^n\woz$ is a $C^\infty$-diffeomorphism . Moreover, the metric $d_\poin$ is locally bi-Lipschitz to the Euclidean metric on~$D^n$. Hence, for every set $A$, every $m$-plane $V\in G(n,m)$ and every $s>0$, 
$P^\eucl_V(A)$ is an $\Hhh^s$-zero set if and only if $P^\poin_V (A)$ is an $\Hhh^s$- zero set. In particular, it follows that $\dim P^\eucl(A)=\dim P^\poin_V (A)$. Moreover, $P^\eucl_V(A)$ has non-empty interior in $V$ if and only if $P^\poin_V (A)$ has non-empty interior in $V$. Hence, Theorem~\ref{thm_Marstrand} and Theorem~\ref{thm_BesFed} follow from their well-known analogs for orthogonal projection onto $m$-planes in $\R^n$.
\end{proof}

\vspace{0.5cm}

\section{Remark on transversality and projection theorems}\label{sec_remarks}

In \cite{PS2000} Peres and Schlag establish a very general projections theorem for families of (abstract) projections from compact metric spaces to Euclidean space. Namely, their result states that if a sufficiently regular family of projections satisfies a certain transversality condition then this yields bounds for the Sobolev dimension of the push-forward (by the projections) of certain measures. While Peres and Schlag's main applications of this result concern Bernoulli convolutions, all the classical Marstrand-type projection theorems for Euclidean spaces $\R^n$ can be deduced as corollaries from their result; see Section~6 in~\cite{PS2000} and Section~18.3 in~\cite{Mattila2015}.  Moroever, Hovila et.\ al.~\cite{HJJL2012} has proven that if a family of abstract projections satisfies transversality with sufficiently good transversality constants  then a Besicovitch-Federer type characterization of purely unrectifiable sets in terms of this family of projections follows. Therefore, transversality has proven to be a powerful method in establishing Marstrand-type as well as Besicovitch-Federer type projection theorems in various settings. In particular, the works \cite{Hovila2014} (Heisenberg groups) and \cite{BaloghIseli2016} (Riemannian surfaces of constant curvature) are based on Peres and Schlag's notion of transversality.\smallskip

In fact, it is possible to establish transversality for the family of orthogonal projections in the Poincar\'e model, $P^\poin_V:D^n\to \exp_p(V)\cap D^n$, $V\in G(n,m)$. This is worked out in detailed in the second author's PhD thesis~\cite{AnninaPhD}, Section 6.2. The transversality constants obtained (namely $L=2$ and $\delta=0$ in the notation of \cite{AnninaPhD}) are sufficient to imply both Marstrand-type as well as Besicovitch-Federer-type projection theorems. In particular, Theorem~\ref{thm_BesFed} can be deduced as a corollary from this result.  However, the upper bounds for the dimension of the exceptional set of planes for Marstrand-type projection theorems
in general depend on the transversality constants; see Theorem 7.3 in \cite{PS2000}. In particular, in the cases where $\dim A>m$, the bounds obtained by establishing transversality are worse that the bounds in Conclusions (2.b) and (3.b) of Theorem \ref{thm_Marstrand}. The transversality constants $L=2$ and $\delta=0$ obtained in \cite{AnninaPhD} could still be improved. In particular, a lengthly but straight-forward calculation shows that $L$ can be improved to $3$. However, in order to obtain Theorem \ref{thm_Marstrand} as a consequence of transversality, one would need $L=\infty$. However, this is not possible due to insufficient regularity of the mapping $\Psi:D^n\to D^n$ in the origin.
\newpage

\bibliographystyle{abbrv}
\bibliography{literature_projections}

\end{document}